\documentclass[a4paper,fleqn]{cas-sc}

\usepackage[authoryear,longnamesfirst]{natbib}
\usepackage{amsmath,amssymb,mathtools}
\usepackage{booktabs}
\usepackage{tikz}
\usepackage{xspace}

\newtheorem{theorem}{Theorem}[section]
\newtheorem{lemma}[theorem]{Lemma}
\newtheorem{corollary}[theorem]{Corollary}
\newtheorem{proposition}[theorem]{Proposition}
\newdefinition{definition}[theorem]{Definition}
\newproof{pf}{Proof}

\begin{document}
\let\WriteBookmarks\relax
\def\floatpagepagefraction{1}
\def\textpagefraction{.001}

\shorttitle{Pairwise compatibility of grid graphs}
\shortauthors{Hakim and Bayzid}

\title[mode=title]{Pairwise Compatibility Representations of Multidimensional Grid Graphs}

\author[1]{Sheikh Azizul Hakim}
\ead{hakim@cse.buet.ac.bd}
\credit{Conceptualization, Methodology, Programming based Analysis and  Investigation, Writing -- original draft, Visualization}

\author[1]{Md. Shamsuzzoha Bayzid}
\cormark[1]
\ead{shams_bayzid@cse.buet.ac.bd}
\credit{Conceptualization, Methodology, Supervision, Validation, Writing -- review and editing}

\affiliation[1]{organization={Department of Computer Science and Engineering, Bangladesh University of Engineering and Technology (BUET)},
            city={Dhaka},
            postcode={1000},
            country={Bangladesh}}

\cortext[1]{Corresponding author}

\begin{abstract}
Pairwise compatibility graphs (PCGs) represent graph adjacency by an interval of leaf-to-leaf distances in a weighted tree. We study  grid graphs under the PCG model and two natural extensions: multi-interval PCGs and OR-PCGs. First, we prove that every $d$-dimensional grid graph is a $(d-1)$-interval-PCG. The construction decomposes the grid into hyperplanes of constant coordinate sum and uses a large-base encoding so that distances between consecutive hyperplanes identify the coordinate direction of an edge. A pair of nearby code values is then merged into one interval, reducing the number of intervals from $d$ to $d-1$. Second, we prove that every $d$-dimensional grid is a $\lceil d/2\rceil$-OR-PCG by grouping coordinate directions into pairs; each paired-direction graph is a disjoint union of two-dimensional grid graphs and is therefore a PCG. Finally, an exact tree-metric satisfiability computation shows that $P_3\square P_3\square P_3$ is not a PCG. Consequently, the minimum number of intervals sufficient for all three-dimensional grid graphs is exactly two, resolving a previously posed open problem. The same obstruction shows that the OR-PCG bound is tight in dimension three and implies that every grid with at least three factors of order at least three is not a PCG.
\end{abstract}



\maketitle

\section{Introduction}\label{sec:introduction}

Pairwise compatibility graphs (PCGs) encode graph adjacency through distances among the leaves of a weighted tree. A graph $G=(V,E)$ is a PCG if there exist a positively edge-weighted tree whose leaves are in bijection with $V$ and an interval $[d_{\min},d_{\max}]$ such that two vertices are exactly adjacent when the distance between their corresponding leaves belongs to that interval. The model arose from questions involving phylogenetic trees and has since developed into an independent graph-representation framework; see, for example, \citep{Yanhaona2009PCG,CalamoneriSinaimeri2016Survey,Rahman01092020, CalamoneriMontiSinaimeri2025Variants}.

The compact definition of PCGs hides a difficult recognition problem. A witness consists simultaneously of an unknown tree topology, positive edge weights, and two interval endpoints. Consequently, even highly structured graph families can be nontrivial to classify. This has motivated the study of both special PCG subclasses and natural superclasses. Two extensions are particularly relevant here. A $k$-interval-PCG retains one weighted tree but permits a union of $k$ disjoint admissible distance intervals \citep{AhmedRahman2017MultiInterval}. A $k$-OR-PCG is the union of $k$ PCGs on the same vertex set \citep{CalamoneriMontiSinaimeri2021Generalizations}. Every $k$-interval-PCG is therefore a $k$-OR-PCG, although the converse need not hold.

Grid graphs form a natural test family for these models. They are Cartesian products of paths, have a highly local adjacency rule, and occur throughout graph theory and discrete mathematics. Two-dimensional grid graphs are PCGs \citep{HakimPapanRahman2022}, while a proper subclass of three-dimensional grid graphs was previously shown to admit two-interval representations \citep{PapanPrantoRahman2023}. It remains an open problem to find the minimum integer $k$ such that every three-dimensional grid is a $k$-interval-PCG \citep{CalamoneriMontiSinaimeri2025Variants}.

We resolve this question and establish general bounds in arbitrary dimension. Our contributions are as follows.
\begin{enumerate}
    \item Every $d$-dimensional  grid is a $(d-1)$-interval-PCG. The proof uses hyperplanes of constant coordinate sum and a large-base coordinate encoding on a single weighted tree.
    \item Every $d$-dimensional  grid is a $\lceil d/2\rceil$-OR-PCG. The proof partitions the coordinate directions into pairs and represents each paired-direction subgraph separately.
    \item The three-dimensional grid $P_3\square P_3\square P_3$ is not a PCG. This is established through an exact satisfiability formulation based on the four-point characterization of finite tree metrics.
\end{enumerate}

The third result supplies matching lower bounds in dimension three. Hence the smallest $k$ for which every three-dimensional grid is a $k$-interval-PCG is exactly $2$, and the smallest $k$ for which every three-dimensional grid is a $k$-OR-PCG is also exactly $2$. Moreover, since PCGs are hereditary under induced subgraphs ~\cite{Yanhaona2009PCG}, the $3\times3\times3$ obstruction implies that every grid with at least three path factors of order at least three is not a PCG.

The paper is organized as follows. Section~\ref{sec:preliminaries} introduces the required definitions. Section~\ref{sec:interval} gives the $(d-1)$-interval construction. Section~\ref{sec:or} proves the OR-PCG bound. Section~\ref{sec:tightness} establishes tightness in dimension three and derives the induced-subgraph consequences. Section~\ref{sec:conclusion} concludes with open questions.

\section{Preliminaries}\label{sec:preliminaries}

All graphs considered in this paper are finite, simple, and undirected. A weighted tree always has strictly positive real edge weights. For a weighted tree $T$ and vertices $x,y\in V(T)$, let $d_T(x,y)$ denote the total weight of the unique $x$--$y$ path in $T$. We write $L(T)$ for the set of leaves of $T$.

\begin{definition}[Grid graph]
For positive integers $n_1,\ldots,n_d$, the $d$-dimensional  grid graph is
$G=P_{n_1}\square P_{n_2}\square\cdots\square P_{n_d}$.
Its vertex set is identified with
$V(G)={(x_1,\ldots,x_d):0\le x_i\le n_i-1 \text{ for every } i\in{1,\ldots,d}}$.
Two vertices are adjacent if and only if they differ by $1$ in exactly one coordinate and agree in every other coordinate.
\end{definition}

\begin{definition}[Pairwise compatibility graph]
A graph $G=(V,E)$ is a \emph{pairwise compatibility graph}, abbreviated PCG, if there exist a positively weighted tree $T$, a bijection $\zeta:V\rightarrow L(T)$, and real numbers $0\le d_{\min}\le d_{\max}$ such that, for every pair of distinct vertices $u,v\in V$,
$uv\in E$ if and only if
$d_{\min}\le d_T(\zeta(u),\zeta(v))\le d_{\max}$.
\end{definition}

When the bijection $\zeta$ is clear from context, we identify each vertex $v\in V$ with its corresponding leaf $\zeta(v)\in L(T)$ and write $d_T(u,v)$ instead of $d_T(\zeta(u),\zeta(v))$.

\begin{definition}[$k$-interval-PCG]
Let $k\ge1$. A graph $G=(V,E)$ is a \emph{$k$-interval-PCG} if there exist a positively weighted tree $T$, a bijection $\zeta:V\rightarrow L(T)$, and pairwise disjoint closed intervals $I_1,\ldots,I_k$ such that, for every pair of distinct vertices $u,v\in V$,
$uv\in E$ if and only if
$d_T(\zeta(u),\zeta(v))\in\bigcup_{j=1}^{k}I_j$.
The case $k=1$ is exactly the class of PCGs.
\end{definition}

\begin{definition}[$k$-OR-PCG]
Let $k\ge1$. A graph $G=(V,E)$ is a \emph{$k$-OR-PCG} if there exist PCGs $G_i=(V,E_i)$, for $1\le i\le k$, defined on the same vertex set $V$, such that
$E=E_1\cup E_2\cup\cdots\cup E_k$.
\end{definition}

The following standard closure property will be used in Section~\ref{sec:or}.

\begin{lemma}\label{lem:disjoint-copies}~\citep{XiaoNagamochi2020Reductions}
If $H$ is a PCG, then every finite disjoint union of copies of $H$ is also a PCG.
\end{lemma}

\section{A $(d-1)$-interval representation}\label{sec:interval}

We first give a uniform single-tree construction for grid graphs of arbitrary dimension. Two-dimensional grid graphs are PCGs \citep{HakimPapanRahman2022}, so the substantive case is $d\ge3$.

\begin{theorem}\label{thm:dminus1}
For every integer $d\ge3$ and all positive integers $n_1,\ldots,n_d$, the grid
\(
P_{n_1}\square P_{n_2}\square\cdots\square P_{n_d}
\)
is a $(d-1)$-interval-PCG.
\end{theorem}

\begin{pf}
Let
\(
G=P_{n_1}\square\cdots\square P_{n_d} 
\) and 
\(
V(G)=\{(x_1,\ldots,x_d):0\le x_i\le n_i-1\}.
\)
Set
\(
R=\max_{1\le i\le d}(n_i-1)
\)
and choose an integer $M$ such that
\(
M>dR+2.
\)
For $1\le i\le d$, define
\[
W_i=
\begin{cases}
M^{d-i}, & 1\le i\le d-2,\\
M^2-1, & i=d-1,\\
M, & i=d.
\end{cases}
\]
For a grid vertex $x=(x_1,\ldots,x_d)$, let
\(
\phi(x)=\sum_{i=1}^{d}W_i x_i.
\)
Since the origin belongs to $V(G)$, the minimum value of $\phi$ is $0$. Define
\(
\Lambda=1+\max_{x\in V(G)}\phi(x),
\)
so that $0\le\phi(x)\le\Lambda-1$ for every $x$.

For each integer $s$, let
\(
\mathcal H_s=\left\{x\in V(G):\sum_{i=1}^{d}x_i=s\right\}.
\)
If $L=\sum_{i=1}^{d}(n_i-1)$, then the nonempty hyperplanes are $\mathcal H_0,\ldots,\mathcal H_L$. Every grid edge joins two consecutive hyperplanes.

We now construct a weighted tree. For every $s\in\{0,\ldots,L\}$, create a segment $\alpha_s\beta_s$ of length $\Lambda$. For $s<L$, join $\beta_s$ to $\alpha_{s+1}$ by an edge of length
\(
H=M^{d-1}+2.
\)
Thus the backbone is
\(
\alpha_0-\beta_0-\alpha_1-\beta_1-\cdots-\alpha_L-\beta_L.
\)
For every vertex $v\in\mathcal H_s$, place a point $q_v$ on $\alpha_s\beta_s$ at distance $\phi(v)$ from $\alpha_s$, and attach a leaf $\ell_v$ to $q_v$ by an edge of length $1$.

\begin{figure}[pos=htbp]
\centering
\begin{tikzpicture}[scale=.88,every node/.style={font=\small}]
  \coordinate (a0) at (0,0);
  \coordinate (b0) at (4,0);
  \coordinate (a1) at (6.4,0);
  \coordinate (b1) at (10.4,0);
  \draw[thick] (a0)--(b0)--(a1)--(b1);
  \node[below] at (a0) {$\alpha_s$};
  \node[below] at (b0) {$\beta_s$};
  \node[below] at (a1) {$\alpha_{s+1}$};
  \node[below] at (b1) {$\beta_{s+1}$};
  \coordinate (qu) at (1.35,0);
  \coordinate (qv) at (8.2,0);
  \fill (qu) circle (1.5pt);
  \fill (qv) circle (1.5pt);
  \coordinate (lu) at (1.35,1.55);
  \coordinate (lv) at (8.2,1.55);
  \draw[thick] (qu)--(lu);
  \draw[thick] (qv)--(lv);
  \fill (lu) circle (1.7pt);
  \fill (lv) circle (1.7pt);
  \node[above] at (lu) {$\ell_u$};
  \node[above] at (lv) {$\ell_v$};
  \node[below=2pt] at (qu) {$q_u$};
  \node[below=2pt] at (qv) {$q_v$};
  \draw[<->] (0,-.78)--(4,-.78);
  \node at (2,-1.12) {length $\Lambda$};
  \draw[<->] (4,-.78)--(6.4,-.78);
  \node at (5.2,-1.12) {length $H$};
  \draw[<->] (6.4,-.78)--(10.4,-.78);
  \node at (8.4,-1.12) {length $\Lambda$};
\end{tikzpicture}
\caption{The tree construction for two consecutive hyperplanes.}
\label{fig:tree-construction}
\end{figure}
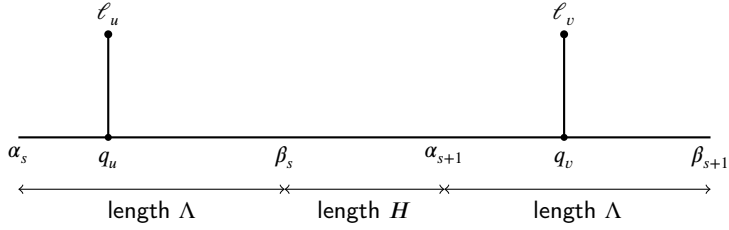

Let
\(
K=H+\Lambda+2.
\)
We define $d-1$ pairwise disjoint intervals. For $1\le j\le d-3$, let
\(
I_j=\left[K+M^{d-j}-\frac{1}{10},\ K+M^{d-j}+\frac{1}{10}\right].
\)
The next interval contains the two consecutive code values $M^2-1$ and $M^2$:
\(
I_{d-2}=\left[K+M^2-\frac{11}{10},\ K+M^2+\frac{1}{10}\right].
\)
Finally, let
\(
I_{d-1}=\left[K+M-\frac{1}{10},\ K+M+\frac{1}{10}\right]
\)
and write $\mathcal I=\bigcup_{j=1}^{d-1}I_j$. Since $M\ge4$, these intervals are pairwise disjoint.

We determine the leaf distances. Let $u\in\mathcal H_s$ and $v\in\mathcal H_{s+k}$, where $k\ge0$.
If $k=0$, then
\(
d_T(\ell_u,\ell_v)=2+|\phi(v)-\phi(u)|\le\Lambda+1.
\)
If $k=1$, the path crosses exactly one bridge, and
\[
d_T(\ell_u,\ell_v)
 =1+(\Lambda-\phi(u))+H+\phi(v)+1
 =K+\phi(v)-\phi(u).
\]
If $k\ge2$, at least two bridges are crossed, and the extreme possible positions of $q_u$ and $q_v$ give
\(
d_T(\ell_u,\ell_v)\ge2H+\Lambda+3.
\)

Suppose first that $uv\in E(G)$ with $u\in\mathcal H_s$ and $v\in\mathcal H_{s+1}$. If the $i$th coordinate increases, then
\(
\phi(v)-\phi(u)=W_i,
\)
and therefore
\(
d_T(\ell_u,\ell_v)=K+W_i\in\mathcal I.
\)
Indeed, the values $W_1,\ldots,W_{d-3}$ are the centers of $I_1,\ldots,I_{d-3}$, the two values $W_{d-2}=M^2$ and $W_{d-1}=M^2-1$ both lie in $I_{d-2}$, and $W_d=M$ is the center of $I_{d-1}$.

It remains to exclude nonedges. Pairs in the same hyperplane have distance at most $\Lambda+1$, which is below the smallest point of $\mathcal I$. Pairs whose hyperplane indices differ by at least two have distance at least $2H+\Lambda+3$. On the other hand, the largest point of $\mathcal I$ is
\(
K+M^{d-1}+\frac{1}{10}
 =2M^{d-1}+\Lambda+4+\frac{1}{10},
\)
whereas
\(
2H+\Lambda+3=2M^{d-1}+\Lambda+7.
\)
Hence these distances lie above $\mathcal I$.

Consider finally a nonedge $uv$ with $u\in\mathcal H_s$ and $v\in\mathcal H_{s+1}$. Write
\(
\Delta_i=v_i-u_i.
\)
Then
\(
\sum_{i=1}^{d}\Delta_i=1,
\qquad
|\Delta_i|\le R,
\)
and $u,v$ are adjacent exactly when $(\Delta_1,\ldots,\Delta_d)$ is a standard unit vector. Suppose, toward a contradiction, that $d_T(\ell_u,\ell_v)\in\mathcal I$. The distance is an integer, and the only integers contained in $\mathcal I$ are $K+W_1,\ldots,K+W_d$. Thus
\(
\sum_{i=1}^{d}W_i\Delta_i=W_t
\)
for some $t\in\{1,\ldots,d\}$.

We use the following elementary uniqueness fact.

\begin{lemma}\label{lem:base-uniqueness}
Let $a_0,\ldots,a_q$ be integers with $|a_i|<M$. If
\(
a_0+a_1M+\cdots+a_qM^q=0,
\)
then $a_0=\cdots=a_q=0$.
\end{lemma}

\begin{pf}
Reducing the equality modulo $M$ gives $a_0\equiv0\pmod M$. Since $|a_0|<M$, we have $a_0=0$. Divide by $M$ and repeat.
\end{pf}

Expanding the code equation gives
\(
-\Delta_{d-1}+M\Delta_d+M^2(\Delta_{d-2}+\Delta_{d-1})
   +\sum_{i=1}^{d-3}M^{d-i}\Delta_i=W_t.
\)
Write $W_t$ in the same base-$M$ form: it is $-1+M^2$ when $t=d-1$, $M$ when $t=d$, $M^2$ when $t=d-2$, and $M^{d-t}$ otherwise. After moving the right-hand side to the left, every coefficient has absolute value at most $2R+1$. Because $d\ge3$ and $M>dR+2$, we have
\(
2R+1<M.
\)
Lemma~\ref{lem:base-uniqueness} therefore forces equality of the coefficients of every power of $M$. The constant coefficient gives
\(
\Delta_{d-1}=\begin{cases}1,&t=d-1,\\0,&t\ne d-1,\end{cases}
\)
the coefficient of $M$ gives
\(
\Delta_d=\begin{cases}1,&t=d,\\0,&t\ne d,\end{cases}
\)
and the coefficient of $M^2$, together with the already determined value of $\Delta_{d-1}$, gives
\(
\Delta_{d-2}=\begin{cases}1,&t=d-2,\\0,&t\ne d-2.\end{cases}
\)
For every $j\le d-3$, comparison of the coefficient of $M^{d-j}$ yields
\(
\Delta_j=\begin{cases}1,&t=j,\\0,&t\ne j.\end{cases}
\)
Thus $(\Delta_1,\ldots,\Delta_d)$ is the standard unit vector in direction $t$, contradicting the assumption that $uv$ is a nonedge.

We have shown that
\(
uv\in E(G)
\quad\Longleftrightarrow\quad
 d_T(\ell_u,\ell_v)\in\mathcal I.
\)
Since $\mathcal I$ is the union of $d-1$ pairwise disjoint intervals, $G$ is a $(d-1)$-interval-PCG.
\end{pf}

\begin{corollary}\label{cor:3d-two-interval}
Every three-dimensional grid is a $2$-interval-PCG.
\end{corollary}

\section{A $\lceil d/2\rceil$-OR representation}\label{sec:or}

The interval construction uses one tree to encode all coordinate directions. A simpler decomposition gives a complementary upper bound when several PCG predicates may be combined by union.

\begin{theorem}\label{thm:or-bound}
For every integer $d\ge1$ and all positive integers $n_1,\ldots,n_d$, the grid
\(
P_{n_1}\square P_{n_2}\square\cdots\square P_{n_d}
\)
is a $\lceil d/2\rceil$-OR-PCG.
\end{theorem}

\begin{pf}
Let $m=\lceil d/2\rceil$ and partition the coordinate indices into blocks
\(
B_j=\{2j-1,2j\}\cap\{1,\ldots,d\},\ 1\le j\le m.
\)
For each $j$, define a graph $G_j$ on the vertex set of the original grid by retaining exactly those grid edges whose changing coordinate lies in $B_j$.

Fixing all coordinates outside $B_j$ produces one connected component of $G_j$. If $|B_j|=2$, that component is isomorphic to
\(
P_{n_{2j-1}}\square P_{n_{2j}};
\)
if $|B_j|=1$, it is isomorphic to $P_{n_{2j-1}}\cong P_{n_{2j-1}}\square P_1$. Every such component is therefore a two-dimensional grid, possibly degenerate, and hence is a PCG by \citet{HakimPapanRahman2022}. All components of $G_j$ are mutually isomorphic, so Lemma~\ref{lem:disjoint-copies} implies that $G_j$ is a PCG.

Every edge of the full grid changes exactly one coordinate, which belongs to exactly one block $B_j$. Therefore
\(
E(G)=E(G_1)\cup\cdots\cup E(G_m).
\)
Thus $G$ is an $m$-OR-PCG, where $m=\lceil d/2\rceil$.
\end{pf}

\section{Tightness in dimension three}\label{sec:tightness}

Both general bounds equal $2$ when $d=3$. We show that neither can be reduced to $1$ for the full family of three-dimensional grid graphs.

\subsection{The $3\times3\times3$ grid}

Let
\(
G_{3,3,3}=P_3\square P_3\square P_3.
\)
Its vertex set is $\{0,1,2\}^3$, and it has $27$ vertices and $54$ edges.

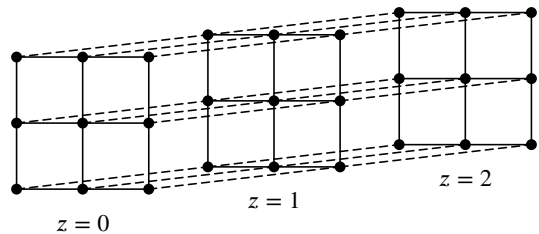
\begin{figure}[pos=htbp]
\centering
\begin{tikzpicture}[
    scale=.92,
    line cap=round,
    line join=round,
    gridvertex/.style={circle,fill=black,inner sep=1.45pt},
    gridedge/.style={semithick},
    layeredge/.style={semithick,densely dashed}
]
\def\s{.95}
\def\dx{2.75}
\def\dy{.32}

\foreach \z in {0,1,2}{
  \foreach \x in {0,1,2}{
    \foreach \y in {0,1,2}{
      \coordinate (v-\x-\y-\z) at ({\s*\x+\dx*\z},{\s*\y+\dy*\z});
    }
  }
}

\foreach \z in {0,1}{
  \pgfmathtruncatemacro{\zp}{\z+1}
  \foreach \x in {0,1,2}{
    \foreach \y in {0,1,2}{
      \draw[layeredge] (v-\x-\y-\z)--(v-\x-\y-\zp);
    }
  }
}

\foreach \z in {0,1,2}{
  \foreach \x in {0,1}{
    \pgfmathtruncatemacro{\xp}{\x+1}
    \foreach \y in {0,1,2}{
      \draw[gridedge] (v-\x-\y-\z)--(v-\xp-\y-\z);
    }
  }
  \foreach \x in {0,1,2}{
    \foreach \y in {0,1}{
      \pgfmathtruncatemacro{\yp}{\y+1}
      \draw[gridedge] (v-\x-\y-\z)--(v-\x-\yp-\z);
    }
  }
}

\foreach \z in {0,1,2}{
  \foreach \x in {0,1,2}{
    \foreach \y in {0,1,2}{
      \node[gridvertex] at (v-\x-\y-\z) {};
    }
  }
}

\node[font=\small] at ({\s},{-.48}) {$z=0$};
\node[font=\small] at ({\s+\dx},{-.48+\dy}) {$z=1$};
\node[font=\small] at ({\s+2*\dx},{-.48+2*\dy}) {$z=2$};
\end{tikzpicture}
\caption{The grid $P_3\square P_3\square P_3$, displayed as three horizontally separated $3\times3$ layers. Dashed edges join corresponding vertices in consecutive layers.}
\label{fig:grid333}
\end{figure}

\subsection{Exact tree-metric verification}\label{subsec:exact-verification}

We briefly recall the exact recognition principle used for the lower bound. A finite metric $d$ is a tree metric if and only if it satisfies the four-point condition: for every four distinct elements $a,b,c,d$, the maximum of
\(
d(a,b)+d(c,d),\qquad d(a,c)+d(b,d),\qquad d(a,d)+d(b,c)
\)
is attained at least twice \citep{Buneman1971,SempleSteel2003}.

For a graph $G=(V,E)$, introduce one positive real variable $d_{uv}$ for every unordered pair of distinct vertices and interval variables $L,U$ with $0\le L\le U$. Impose the metric axioms, the four-point condition for every four-element subset of $V$, and
\(
uv\in E \Longrightarrow L\le d_{uv}\le U,
\)
\(
uv\notin E \Longrightarrow d_{uv}<L\ \lor\ d_{uv}>U.
\)

\begin{proposition}\label{prop:exact-recognition}
The resulting constraint system is satisfiable if and only if $G$ is a PCG.
\end{proposition}

\begin{pf}
A PCG witness directly supplies a satisfying tree metric and interval. Conversely, a satisfying assignment defines a finite metric satisfying the four-point condition, and therefore has a weighted-tree realization. If a labeled point occurs internally rather than as a leaf, attach a new leaf to it by an equal positive pendant length for every label and shift both interval endpoints by twice that length. This preserves the represented graph and produces a valid PCG witness with all graph vertices represented by leaves.
\end{pf}

For $G_{3,3,3}$, the exact instance contains $\binom{27}{2}=351$ pair-distance variables and imposes the four-point condition on all $\binom{27}{4}=17{,}550$ four-element subsets. We implemented the equivalent exact formulation in Z3 \citep{deMouraBjorner2008}. The complete instance was run on a CloudLab Clemson \texttt{r6615} node and terminated after approximately $30$ hours with the result \texttt{unsat}. The computation used exact arithmetic; it was not a floating-point or bounded-weight search. The use of the four-point characterization for exact SMT-based recognition and enumeration of PCGs is also developed in our concurrent work~\citep{HakimSultanaBayzid2026SMT}.

\begin{theorem}\label{thm:333-not-pcg}
The grid $P_3\square P_3\square P_3$ is not a PCG.
\end{theorem}

\begin{pf}
By Proposition~\ref{prop:exact-recognition}, satisfiability of the exact tree-metric instance is equivalent to the existence of a PCG representation. The instance for $G_{3,3,3}$ is unsatisfiable. Therefore no positively weighted tree and single distance interval represent $P_3\square P_3\square P_3$.
\end{pf}


\subsection{Consequences}

\begin{corollary}\label{cor:interval-tight}
The minimum integer $k$ such that every three-dimensional grid is a $k$-interval-PCG is exactly $2$.
\end{corollary}

\begin{pf}
Corollary~\ref{cor:3d-two-interval} gives the upper bound $2$. Theorem~\ref{thm:333-not-pcg} gives a three-dimensional grid that is not a $1$-interval-PCG, since $1$-interval-PCGs are precisely PCGs.
\end{pf}

\begin{corollary}\label{cor:or-tight}
The minimum integer $k$ such that every three-dimensional grid is a $k$-OR-PCG is exactly $2$.
\end{corollary}

\begin{pf}
Theorem~\ref{thm:or-bound} gives the upper bound $2$. A $1$-OR-PCG is a PCG, and Theorem~\ref{thm:333-not-pcg} supplies the matching lower bound.
\end{pf}

\begin{corollary}\label{cor:large-grid graphs}
Let
\(
G=P_{n_1}\square\cdots\square P_{n_d}.
\)
If at least three of $n_1,\ldots,n_d$ are at least $3$, then $G$ is not a PCG. In particular, $P_a\square P_b\square P_c$ is not a PCG whenever $a,b,c\ge3$.
\end{corollary}

\begin{pf}
Choose three path factors of order at least $3$, select three consecutive vertices in each, and fix one arbitrary vertex in every other factor. The selected vertices induce $P_3\square P_3\square P_3$. The PCG class is hereditary under induced subgraphs: deleting leaves outside a selected vertex set preserves all distances among the remaining leaves. Hence a PCG containing this induced subgraph would imply that $P_3\square P_3\square P_3$ is a PCG, contradicting Theorem~\ref{thm:333-not-pcg}.
\end{pf}

\section{Concluding remarks}\label{sec:conclusion}

We established two general upper bounds for grid graphs under natural extensions of pairwise compatibility graphs. The $(d-1)$-interval construction keeps one weighted tree and uses a large-base encoding to distinguish coordinate directions. The $\lceil d/2\rceil$-OR construction instead partitions the directions into pairs and combines simpler PCG representations. These approaches exploit different aspects of the Cartesian-product structure.

The exact obstruction $P_3\square P_3\square P_3$ determines the worst-case parameters in dimension three: two intervals are necessary and sufficient for all three-dimensional grid graphs, and two PCG predicates are necessary and sufficient in the OR model. The induced-subgraph consequence shows more broadly that every grid with three factors of order at least three lies outside the PCG class.

Several questions remain open. The obstruction above does not classify thin three-dimensional grid graphs such as $P_2\square P_m\square P_n$. For $d\ge4$, the gap between the interval upper bound $d-1$ and the general lower bound $2$ remains wide, and the optimality of the $\lceil d/2\rceil$ OR bound is also unknown. A structural proof of the $3\times3\times3$ obstruction would be especially valuable, both for conceptual understanding and for identifying broader forbidden configurations.

\section*{Data and code availability}
The source code is available at \url{https://github.com/skhakim/PCG-Enumerative-Study/tree/grid333-check}.

\section*{Acknowledgement}
We are grateful to Muhammad Abdullah Adnan for facilitating our access to CloudLab. We used generative AI tools for language polishing and code refactoring assistance.


\bibliographystyle{cas-model2-names}
\bibliography{grid-pcg-refs}

\end{document}